\documentclass[12pt,letterpaper]{article}
\usepackage[margin=35mm]{geometry}

\usepackage{amsmath}
\usepackage{amssymb}
\usepackage{microtype}
\usepackage{amsthm}

\newtheorem{thm}{Theorem}[section]

\newtheorem{cor}[thm]{Corollary}

\newtheorem{conj}[thm]{Conjecture}

\newtheorem{question}[thm]{Question}

\newcommand{\ga}[0]{\alpha}

\begin{document}
\renewcommand{\thefootnote}{\fnsymbol{footnote}}

\title{Two problems on independent sets in graphs}

\author{David Galvin\thanks{dgalvin1@nd.edu; Department of Mathematics,
University of Notre Dame, Notre Dame IN 46556. Supported in part by National Security Agency grant H98230-10-1-0364.}}

\date{\today}

\maketitle

\begin{abstract}
Let $i_t(G)$ denote the number of independent sets of size $t$ in a graph $G$. Levit and Mandrescu have conjectured that for all bipartite $G$ the sequence $(i_t(G))_{t \geq 0}$ (the {\em independent set sequence} of $G$) is unimodal. We provide evidence for this conjecture by showing that is true for almost all equibipartite graphs. Specifically, we consider the random equibipartite graph $G(n,n,p)$, and show that for any fixed $p\in(0,1]$ its independent set sequence is almost surely unimodal, and moreover almost surely log-concave except perhaps for a vanishingly small initial segment of the sequence. We obtain similar results for $p=\tilde{\Omega}(n^{-1/2})$.

We also consider the problem of estimating $i(G)=\sum_{t \geq 0} i_t(G)$ for $G$ in various families. We give a sharp upper bound on the number of independent sets in an $n$-vertex graph with minimum degree $\delta$, for all fixed $\delta$ and sufficiently large $n$. Specifically, we show that the maximum is achieved uniquely by $K_{\delta, n-\delta}$, the complete bipartite graph with $\delta$ vertices in one partition class and $n-\delta$ in the other.

We also present a weighted generalization: for all fixed $x>0$ and $\delta >0$, as long as $n=n(x,\delta)$ is large enough, if $G$ is a graph on $n$ vertices with minimum degree $\delta$ then $\sum_{t \geq 0} i_t(G)x^t \leq \sum_{t \geq 0} i_t(K_{\delta, n-\delta})x^t$ with equality if and only if $G=K_{\delta, n-\delta}$.
\end{abstract}

\section{Introduction}

For a (simple, loopless, finite) graph $G=(V,E)$, let $i_t(G)$ denote the number of independent sets (or stable sets) of size $t$ in $G$, that is, the number of subsets of $V$ of size $t$ that do not span an edge. This note considers two problems concerning the quantity $i_t(G)$.

The first problem concerns the unimodality and log-concavity of the sequence $(i_t(G))_{t=0}^{\alpha(G)}$, where $\alpha(G)$ is the size of the largest independent set in $G$. We refer to this sequence as the {\em independent set sequence} of $G$.

A sequence $(a_t)_{t = m}^n$ is said to be {\em unimodal} (with mode $k$) if there is some $m \leq k \leq n$ such that
$$
a_m \leq a_{m+1} \leq \ldots \leq a_{k-1} \leq a_k \geq a_{k+1} \geq \ldots \geq a_n.
$$
It is said to be {\em log-concave} if for each $m < k < n$ we have $a_k^2 \geq a_{k-1}a_{k+1}$. It is said to have the {\em real roots} property if the polynomial $\sum_{t=m}^n a_tx^t$ has only real roots. Unimodality of a sequence of positive terms does not imply log-concavity, and log-concavity does not imply real roots, but via a theorem of Newton real roots implies log-concavity and log-concavity is easily seen to imply unimodality.

For certain families of graphs, the independent set sequence is known to be unimodal. A result of Heilmann and Lieb \cite{HeilmannLieb} on the matching polynomial of a graph implies that if $G$ is a line graph then the independent set sequence of $G$ has the real roots property and so is both unimodal and log-concave. Chudnovsky and Seymour \cite{ChudnovskySeymour} extended this to claw-free graphs (graphs without an induced star on four vertices). Earlier Hamidoune \cite{Hamidoune} had shown that the independent set sequence of a claw free graph is both unimodal and log-concave.

In general, however, the independent set sequence of a graph is neither log-concave nor unimodal, as evidenced by the graph $G=3K_4+K_{37}$ (the join of the complete graph on $37$ vertices and the disjoint union of three complete graphs on four vertices each), which has independent set sequence $(1, 49, 48, 64)$. In fact, much more is true. Alavi, Erd\H{o}s, Malde and Schwenk \cite{AlaviErdosMaldeSchwenk} showed that the independent set sequence can be made to be as far from unimodal as one wishes, in the following sense. Let $\pi$ be any of the $m!$ permutations of $\{1, \ldots, m\}$. There is a graph $G$ with $\alpha(G)=m$ such that
$$
i_{\pi(1)}(G) < i_{\pi(2)}(G) < \ldots < i_{\pi(m)}(G).
$$
By contrast, the number of permutations of $\{1, \ldots, m\}$ that can be realized in this way by a unimodal sequence is shown in \cite{AlaviErdosMaldeSchwenk} to be at most $2^{m-1}$.

Alavi et al. speculated that perhaps trees, and maybe even forests, have unimodal independent set sequences. This is the case for two extreme examples, namely paths and stars.
Levit and Mandrescu \cite{LevitMandrescu} have gone further, conjecturing that if $G$ is a K\"{o}nig-Egerv\'{a}ry graph (a graph in which the size of the largest independent set plus the size of the largest matching equals the number of vertices in the graph) then the independent set sequence of $G$ is unimodal. In particular, since all bipartite graphs are K\"{o}nig-Egerv\'{a}ry, we have the following conjecture which extends Alavi et al.'s speculation.
\begin{conj} \label{conj-unimodalfamilies}
If $G$ is a bipartite graph then the independent set sequence of $G$ is unimodal.
\end{conj}
A partial result in the direction of Conjecture \ref{conj-unimodalfamilies} appears in \cite{LevitMandrescu}, where it is shown that if $G$ is bipartite then the final one third of its independent set sequence is decreasing:
\begin{equation} \label{LevitMandrescu-finalthird}
i_{\lceil (2\ga(G)-1)/3\rceil}(G) \geq i_{\lceil (2\ga(G)-1)/3\rceil+1}(G) \geq \ldots \geq i_{\ga(G)}(G).
\end{equation}

\medskip

A natural object to look at for evidence in favor of Conjecture \ref{conj-unimodalfamilies} is the random equibipartite graph. Let ${\mathcal G}(n,n)$ be the collection of all bipartite graphs with fixed bipartition ${\mathcal E} \cup {\mathcal O}$, where $|{\mathcal E}|=|{\mathcal O}|=n$. Let $G(n,n,p)$ be the probability distribution on ${\mathcal G}(n,n)$ in which each particular graph with $m$ edges is chosen with probability $p^m(1-p)^{n^2-m}$. Equivalently, an element of $G(n,n,p)$ is obtained by selecting each of the $n^2$ pairs in ${\mathcal E}\times {\mathcal O}$ with probability $p$, each selection made independently. When $p=1/2$, $G(n,n,p)$ is uniformly distributed on all bipartite graphs with fixed bipartition ${\mathcal E} \cup {\mathcal O}$. For general $p$, $G(n,n,p)$ concentrates around graphs with average degree around $np$. This model was introduced by Erd\H{o}s and R\'enyi in 1964 \cite{ErdosRenyi2}, and there are now many papers devoted to its properties (see e.g. \cite{AlonKrivelevich}, \cite{ErdosRubinTaylor}, \cite{Frieze}).

The random equibipartite graph, along with the family of regular bipartite graphs, was briefly considered in \cite{Galvin-Qdfixed}. Among other results, it was shown that $G(n,n,p)$ has a long equispaced unimodal subsequence. Specifically, for all $p=\omega(1)/n$ there is a function $s(n)=o(n)$ such that for all $\varepsilon > 0$, with probability tending to $1$ as $n$ goes to infinity we have that if $\varepsilon n \leq k \leq \ell \leq (1-\varepsilon)n/2$ and $\ell-k \geq s(n)$ then $i_k(G(n,n,p)) \leq i_\ell(G(n,n,p))$, while if $(1+\varepsilon)n/2 \leq k \leq \ell \leq n$ and $\ell-k \geq s(n)$ then $i_k(G(n,n,p)) \geq i_\ell(G(n,n,p))$.
%
%
The first aim of this note is to extend this observation considerably, and thereby provide further evidence in favor of Conjecture \ref{conj-unimodalfamilies}. In Section \ref{sec-proof} we prove the following result.
\begin{thm} \label{thm-randomunimodality}
Let $G^p_n$ be a graph selected from the distribution $G(n,n,p)$.
\begin{enumerate}
\item \label{list1} Fix $\delta > 0$, and let $p$ satisfy $p \geq \delta$. Almost surely (with probability tending to $1$ as $n$ tends to infinity) the independent set sequence of $G^p_n$ is unimodal with mode $n/2$. Moreover, there is a constant $C=C(p)$ such that almost surely the sequence $(i_t(G^p_n))_{t=[C\log n]}^n$ is log-concave.
\item \label{list2} There are constants $C, D$ such that if $p\geq Dn^{-1/2}\log^{1/2} n$ then almost surely the sequence $(i_t(G^p_n))_{t=[(C\log n)/p]}^n$ is unimodal with mode $n/2$, and moreover is log-concave.
\item \label{list3} If $p \geq (\log n + \log\log n +\omega(1))/n$ then almost surely the sequence $(i_t(G^p_n))_{t=0}^k$ is increasing, where $k=\log n -2\log\log n$.
\end{enumerate}
\end{thm}
\begin{cor}
Fix $\varepsilon > 0$. If $p=\omega(n^{-1/2}\log^{1/2}n)$ then almost surely the sequence $(i_t(G^p_n))_{t=[\varepsilon n]}^n$ is unimodal with mode $n/2$, and moreover is log-concave.
\end{cor}

The following conjecture is weaker than Conjecture \ref{conj-unimodalfamilies}, but is perhaps more approachable and would provide strong evidence in favor of the full conjecture.
\begin{conj}
For all $p=p(n)$, almost surely the sequence $(i_t(G^p_n))_{t=0}^n$ is unimodal, and moreover is log-concave.
\end{conj}

\medskip

The second problem that we address in this note concerns the quantity $i(G)=\sum_{t \geq 0} i_t(G)$, or the number of independent sets in $G$. This quantity was first explicitly considered by Prodinger and Tichy \cite{ProdingerTichy}, and is often referred to as the {\em Fibonacci number} of $G$, motivated by the observation that $P_n$, the path on $n$ vertices, satisfies $i(P_n)=F_{n+2}$, the $(n+2)$nd Fibonacci number. In the field of molecular chemistry, $i(G)$ is referred to as the {\em Merrifield-Simmons} index of $G$ \cite{MerrifieldSimmons}.

Given a family of graphs ${\mathcal G}$ it is natural to ask what range of values $i(G)$ takes as $G$ runs over ${\mathcal G}$, and which are the extremal $G$ that achieve the maximum and minimum. There are numerous papers devoted to this question, for various choices of ${\mathcal G}$. A sample of the literature includes \cite{ProdingerTichy} (trees), \cite{LinLin} (forests), \cite{PedersenVestergaard} (unicyclic graphs), \cite{CutlerRadcliffe} (graphs with a fixed average degree), \cite{Hua} (graphs with a given number of cut-edges), \cite{Kahn,Zhao} (regular graphs with fixed degree), and \cite{Alekseev} (graphs with a given independence number).

The first of these results that is particularly relevant to the present work is that of
Prodinger and Tichy \cite{ProdingerTichy}, who considered trees on $n$ vertices and showed that the maximum of $i(G)$ is achieved by the star and the minimum by the path. (Here the {\em star} on $n$ vertices is the graph $K_{1,n-1}$, where $K_{a,b}$ indicates the complete bipartite graph with $a$ vertices in one partition class and $b$ in the other.) Since every connected graph contains a spanning tree, and deleting edges from a graph cannot decrease its number of independent sets, it immediately follows that among all connected graphs on $n$ vertices, none has more independent sets than the star. (Trivially the number of independent sets over connected graphs is minimized by the complete graph.)

The second particularly relevant previous result is that of Hua \cite{Hua}, who considered $2$-edge-connected graphs on $n$ vertices and showed that for all $n \geq 6$ the maximum of $i(G)$ for this family is achieved uniquely by $K_{2,n-2}$.

Noting that for $n \geq 2$ being connected implies having minimum degree at least one and that being $2$-edge-connected implies having minimum degree at least $2$, the following result contains the case of connected and $2$-edge connected graphs as a special case (modulo some small values of $n$ not covered below).
\begin{thm} \label{thm-main}
Let $C=\ln 4/(\ln 4-1) \approx 3.588$. Let $G$ be a graph on $n$ vertices with minimum degree at least $\delta$, where $n$ and $\delta$ satisfy
$n \geq (C-1)\delta^2 + (3C-1)\delta + 2$. Either $G=K_{\delta, n-\delta}$ or $i(G) < i(K_{\delta, n-\delta})$.
\end{thm}
\begin{cor}
Fix $\delta > 0$. There is $n(\delta)$ such that for all $n \geq n(\delta)$, among all $n$-vertex graphs with minimum degree at least $\delta$, the unique graph with the most independent sets is $K_{\delta,n-\delta}$.
\end{cor}
Note that as long as $n \geq \delta + 1$, the unique $n$-vertex graph with minimum degree at least $\delta$ with the fewest independent sets is trivially $K_n$, the complete graph on $n$ vertices.

We expect that the relationship between $n$ and $\delta$ can be tightened considerably.
\begin{conj} \label{conj-downto2delta}
If $G$ is a graph on $n$ vertices with minimum degree at least $\delta$, where $n$ and $\delta$ satisfy
$n \geq 2\delta$, then $i(G) \leq i(K_{\delta, n-\delta})$.
\end{conj}
Note that in the regime $n \geq 2\delta$ we can no longer expect $K_{\delta,n-\delta}$ to be the unique maximizer of the independent set count: for example, the $5$-cycle and $K_{2,3}$ both have $11$ independent sets.

Theorem \ref{thm-main} generalizes to a natural statement concerning the independent set or stable set polynomial of $G$, first introduced by Gutman and Harary \cite{GutmanHarary}. This is the polynomial
$$
P(G,x) = \sum_{t \geq 0} i_t(G)x^t
$$
where $i_t(G)$ is the number of independent sets in $G$ of size $t$. (Note that $P(G,1)=i(G)$). We prove the following in Section \ref{Section-proofmain}.
\begin{thm} \label{thm-main-weighted}
Let $G$ be a graph on $n$ vertices with minimum degree at least $\delta$, and let $x>0$ be fixed. If $n, x$ and $\delta$ satisfy
\begin{equation} \label{cond}
n \geq (C_x-1)\delta^2 + ((1-D_x)C_x + 1 + D_x)\delta - D_x
\end{equation}
where
$$
C_x=\frac{\ln (1+x)}{\ln (1+x)- \frac{x}{1+x}}
$$
and
$$
D_x = \frac{2\ln \left(\frac{x}{1+x}\right)}{\ln (1+x)},
$$
then either $G=K_{\delta, n-\delta}$ or $P(G,x) < P(K_{\delta, n-\delta},x)$.
\end{thm}
Note that this reduces to Theorem \ref{thm-main} in the case $x=1$. Routine calculus shows that for each fixed $\delta >0$, the right-hand side of (\ref{cond}) above is monotone decreasing in $x$, tending to infinity as $x$ tends to $0$ and tending to $2\delta$ as $x$ tends to infinity.

The cited results from \cite{ProdingerTichy} and \cite{Hua} are proved by induction, with the case of $2$-edge connected graphs requiring an involved case analysis. By contrast, our proof of Theorem \ref{thm-main-weighted} is quite direct.
Our main tool is a result from \cite{Galvin-asymptot-ind-count} (generalizing a result of Alekseev \cite{Alekseev}).  Let $G$ be an $n$-vertex graph whose largest independent set has size $\alpha$. Then for all $x > 0$,
\begin{equation} \label{Alekseev-bd}
P(G,x) \leq \left(1+\frac{nx}{\alpha}\right)^{\alpha}  \leq (1+x)^{\alpha} e^{\frac{(n-\alpha)x}{1+x}}.
\end{equation}
This allows us to quickly focus attention on graphs with large $\alpha$, where a direct bound on $P(G,x)$ is possible.

\medskip

We have considered which $n$-vertex graphs with minimum degree at least $\delta$ maximize the total number of independent sets. A natural question to consider at this point is what happens when we focus on independent sets of a given size.
\begin{question}
Fix $n$, $\delta$ and $t$ with $n \geq 2$, $1 \leq \delta \leq n-1$ and $0 \leq t \leq n$. What is the maximum value of $i_t(G)$ as $G$ ranges over all $n$-vertex graphs with minimum degree at least $\delta$, and for which graphs is the maximum achieved?
\end{question}
This question is trivial for $t=0$, $1$ and $n$ (with all graphs having the same count in each case). For all other $t$, it is tempting to conjecture that the maximum achieved by $K_{\delta,n-\delta}$, but this is not true, as is easily seen by considering the case $t=2$. Maximizing the number of independent sets of size two is the same as minimizing the number of edges, and it is easy to see that for all fixed $\delta$ and sufficiently large $n$, there are $n$-vertex graphs with minimum degree at least $\delta$ which have fewer edges than $K_{\delta,n-\delta}$.

For $\delta=1$, it turns out that $K_{1,n-1}$ is the unique maximizer for all $3 \leq t \leq n-1$. (A weaker statement than the following, with graphs of minimum degree at least $1$ replaced by trees, appeared in \cite{Wingard}.)
\begin{thm}
Fix $n$ and $t$ with $n \geq 4$ and $3 \leq t \leq n-1$. If $G$ is an $n$-vertex graph with minimum degree at least $1$ and $G \neq K_{1,n-1}$ then
$i_t(G) < i_t(K_{1,n-1})$.
\end{thm}

\begin{proof}
The statement is easy to verify for $t=n-1$, as $K_{1,n-1}$ is the only graph on $n$ vertices with minimum degree at least $1$ that has any independent sets of size $n-1$. For the remainder we proceed by induction on $n+t$, with the base case $n+t=7$ (i.e., $n=4$ and $t=3$) being an instance of the case $t=n-1$.

Now fix $n$ and $t$ with $n\geq 5$ and $3 \leq t \leq n-2$. We consider two cases, the first being when $G$ has a vertex $x$ whose deletion results in a graph without isolated vertices. The independent sets of size $t$ in $G$ that do not include $x$ are independent sets of size $t$ in a graph on $n-1$ vertices with minimum degree at least one. By induction, there are at most ${n-2 \choose t}$ such independent sets, with equality only if $G-x=K_{1,n-2}$. The independent sets of size $t$ in $G$ that include $x$ are independent sets of size $t-1$ in a graph on at most $n-2$ vertices; there are trivially at most ${n-2 \choose t-1}$ such sets, with equality only if $x$ has degree $1$, the neighbor of $x$ is adjacent to all other vertices, and there are no other edges; i.e., if $G=K_{1,n-1}$. It follows that
$$
i_t(G) \leq {n-2 \choose t} + {n-2 \choose t-1} = {n-1 \choose t} = i_t(K_{1,n-1})
$$
with equality only if $G=K_{1,n-1}$.

If $G$ does not have a vertex whose deletion results in a graph without isolated vertices, then $n$ must be even and $G$ must be a disjoint union of $n/2$ edges, so
$$
i_t(G) = {n/2 \choose t}2^t < {n-1 \choose t} = i_t(K_{1,n-1}),
$$
the inequality valid for $t \geq 3$.
\end{proof}

We believe that similar behavior occurs for larger $\delta$.
\begin{conj}
Fix $\delta \geq 2$. There is $C(\delta)$ (perhaps $C(\delta)$ may be taken to be $\delta+2$) such that for all $C(\delta) \leq t \leq n-\delta$ and $n \geq \delta + C(\delta)$, if $G$ is an $n$-vertex graph with minimum degree at least $\delta$ and $G \neq K_{\delta,n-\delta}$ then
$i_t(G) < i_t(K_{\delta,n-\delta})$.
\end{conj}


\section{Proof of Theorem \ref{thm-randomunimodality}} \label{sec-proof}

Let $G$ be a bipartite graph on bipartition classes ${\mathcal E}$ and ${\mathcal O}$ with $|{\mathcal E}|=|{\mathcal O}|=n$. A trivial lower bound on $i_t(G)$ (for $t \geq 1$) is
\begin{equation} \label{general-lb}
i_t(G) \geq 2{n \choose t},
\end{equation}
obtained by considering those independent sets of size $t$ that come either entirely from ${\mathcal E}$ or entirely from ${\mathcal O}$.

To obtain a similar-looking upper bound, we set up some notation. Let $K(G)$ be the smallest integer with the property that for each $A \subseteq {\mathcal E}$ and $B \subseteq {\mathcal O}$ with $|A|=|B|=K(G)$, there is an edge in $G$ joining $A$ and $B$. For each $1 \leq k \leq K(G)$, let $m(k,G)$ be the maximum, over all $A \subseteq {\mathcal E}$ (or ${\mathcal O}$) with $|A|=k$, of the number of vertices in ${\mathcal O}$ (or ${\mathcal E}$) that are not adjacent to $A$.

By definition of $K(G)$, for each independent set $I$ in $G$ we have
$$
\min\{|I \cap {\mathcal E}|, |I \cap {\mathcal O}|\} \leq \min\{K(G),|I|/2\}.
$$
By definition of $m(k,G)$, the number of independent sets of size $t$ with $\min\{|I \cap {\mathcal E}|, |I \cap {\mathcal O}|\} = k$ for some $1 \leq k \leq \min\{K(G),t/2\}$ is at most ${n \choose k}{m(k,G) \choose t-k}$. This leads to the following upper bound on $i_t(G)$ for all $t \geq 0$:
\begin{eqnarray}
i_t(G) & \leq & 2\left({n \choose t} + \sum_{k=1}^{\min\{K(G),t/2\}} {n \choose k}{m(k,G) \choose t-k}\right) \nonumber \\
& \leq & 2(1+x(t)){n \choose t} \label{general-ub}
\end{eqnarray}
where
\begin{eqnarray}
x(t) & = & \sum_{k=1}^{\min\{K(G),t/2\}} \frac{{n \choose k}{m(k,G) \choose t-k}}{{n \choose t}}. \nonumber \\
& = & \sum_{k=1}^{\min\{K(G),t/2\}} {t \choose k} \frac{(m(k,G))_{(t-k)}}{(n-k)_{(t-k)}} \nonumber \\
& \leq & \sum_{k=1}^{\min\{K(G),t/2\}} {t \choose k} \left(\frac{m(k,G)}{n-k}\right)^{t-k} \label{spec-xt-bound.1}
\end{eqnarray}
where $a_{(b)}=a(a-1)\ldots(a-b+1)$ is a falling power, and the final inequality is valid as long as $m(k,G) \leq n-k$ for all $k$ under consideration, which will be the case in our application.

An easy calculation gives that if
\begin{equation} \label{xt-bound-uni}
x(t) \leq \left\{
\begin{array}{ll}
\frac{n-2t-1}{t+1} & \mbox{for $1 \leq t \leq n/2-1$} \\
& \\
\frac{2t-1-n}{n-t+1} & \mbox{for $n/2+1 \leq t \leq n$}
\end{array}
\right.
\end{equation}
then we have
$$
i_t(G) \leq 2(1+x(t)){n \choose t} \leq 2{n \choose t+1} \leq i_{t+1}(G)
$$
for $1 \leq t \leq n/2-1$ and
$$
i_t(G) \leq 2(1+x(t)){n \choose t} \leq 2{n \choose t-1} \leq i_{t-1}(G)
$$
for $n/2+1 \leq t \leq n$. So to show that the sequence $(i_t(G))_{t=0}^n$ is unimodal with mode $n/2$, it is sufficient to establish (\ref{xt-bound-uni}) (trivially $1=i_0(G) \leq i_1(G)=2n$ as long as $n \geq 1$).

Note that the first expression in (\ref{xt-bound-uni}) above is decreasing in $t$ and its minimum in the range $1 \leq t \leq n/2-1$ is at least $1/n$ (as long as $n \geq 2$); similarly, the second term is increasing in $t$ and its minimum in the range $n/2+1 \leq t \leq n$ is also at least $1/n$.

Similarly, for $1 \leq t \leq n-1$ the condition
\begin{equation} \label{xt-bound-logcon}
\left(1 + \frac{1}{t}\right)\left(1 + \frac{1}{n-t}\right) \geq (1+x(t-1))(1+x(t+1))
\end{equation}
is sufficient to establish (via (\ref{general-lb}) and (\ref{general-ub})) the log-concavity condition $i_t(G)^2 \geq i_{t-1}(G)i_{t+1}(G)$.

\medskip

We now turn our attention to $G^p_n$. Set $d=np$. We begin by establishing some almost-sure properties of $G^p_n$.

We begin with the range $d=\omega(1)$.
First, suppose $A \subseteq {\mathcal E}$ and $B \subseteq {\mathcal O}$ satisfy $|A| = |B| = (2n\log d)/d$. Then almost surely there is an edge from $A$ to $B$. Indeed, the probability that there exists a pair $A$, $B$ of the given sizes that fail to satisfy the property is at most
\begin{eqnarray*}
{n \choose (2n \log d)/d}^2 (1-p)^{(4n^2 \log^2d)/d^2} & \leq & \left(\left(\frac{ed}{2\log d}\right)^2 \exp\{-2\log d\}\right)^{(2n\log d)/d}  \\
& = & \left(\frac{e^2}{4\log^2d}\right)^{(2n\log d)/d} = o(1).
\end{eqnarray*}
Next, for each $1 \leq k \leq (2n\log d)/d$, the number of vertices of ${\mathcal E}$ not neighboring a particular subset of ${\mathcal O}$ of size $k$ is a binomial random variable with parameters $n$ and $(1-p)^k$. By standard Chernoff-type estimates (see for example \cite[Appendix A]{AlonSpencer}), it follows that the probability that more than $n(1-p)^k + t$ vertices are not covered is at most $\exp\{-t^2/(3n(1-p)^k)\}$, and so the probability that there is at least one choice of subset of ${\mathcal O}$ of size $k$ that leaves more than $n(1-p)^k + t$ vertices of ${\mathcal E}$ not covered is at most
$$
{n \choose k} e^{-\frac{t^2}{3n(1-p)^k}} \leq \exp\left\{k \log n -\frac{t^2}{3n(1-p)^k}\right\}.
$$
Choosing
$$
t = 3\sqrt{kn(1-p)^k\log n},
$$
this probability is at most $1/n^2$. It follows that almost surely, $G^p_n$ has the following properties for $d=\omega(1)$:
\begin{itemize}
\item $K(G^p_n) \leq (2n\log d)/d$.
\item $m(k,G^p_n) \leq n(1-p)^k + 3\sqrt{kn(1-p)^k\log n}$ for each $1 \leq k \leq K(G^p_n)$.
\end{itemize}
Next, in the range $p=(\log n + \log \log n + c(n))$ for any $c(n)=\omega(1)$, Frieze \cite{Frieze} has shown that $G^p_n$ almost surely has a Hamilton cycle (and so also a perfect matching). Two immediate consequences of this are as follows.
\begin{itemize}
\item $G^p_n$ is connected and so has largest independent set size $n$.
\item For each $k$, $m(k,G^p_n) \leq n-k$.
\end{itemize}
We will assume from now on that $G^p_n$ has all four of these bulleted properties. Also, since we are concerned with limiting values of quantities as $n$ grows, we will assume throughout that $n$ is large enough to support our assertions.

Our first trivial observation is that for $p=(\log n + \log \log n + \omega(1))/n$ we have
$$
x(t) \leq \sum_{k=1}^{\min\{K(G^p_n),t/2\}} {t \choose k} \leq 2^t.
$$
For $t \leq \log n - 2\log \log n$ this is at most $n/\log^2 n$, and so we have (\ref{xt-bound-uni}) for this range of $t$ and $p$.

Next we show that if $\delta$ is a (small) constant and $C'$ a large constant, then for $p \geq \delta$ and $t \leq C'\log n$, the right-hand side of (\ref{spec-xt-bound.1}) is smaller than $\sqrt{n}$. Indeed, for any constant $\delta'>0$ there is $k(\delta')>0$ such that for $k>k(\delta')$ we have
$$
\frac{m(k,G^p_n)}{n-k} \leq \frac{n(1-p)^k + 3\sqrt{kn(1-p)^k\log n}}{n-(2n\log d)/d} \leq \delta'
$$
since for $p > \delta$ we have
$$
3\sqrt{kn(1-p)^k\log n} \leq C(\delta)\sqrt{n} \log n
$$
and
$$
\frac{2n\log d}{d} \leq C(\delta)\log n.
$$
It follows that
\begin{eqnarray*}
x(t) & \leq & \sum_{k=1}^{k(\delta')} t^k + \sum_{k=k(\delta')}^{\min{K(G^p_n),t/2}} {t \choose k} (\delta')^{t-k} \\
& \leq & k(\delta')t^{k(\delta')} + (\delta')^t\left(1+\frac{1}{\delta'}\right)^t \\
& \leq & k(\delta')(C'\log n)^{k(\delta')} + (1+\delta')^{C'\log n}.
\end{eqnarray*}
For suitably small $\delta'$, this is at most $\sqrt{n}$.

For the remainder, we parameterize by $p=f(n)/n$, so $d=f(n)$. Note that we are always assuming $f(n)\geq \log n + \log\log n + \omega(1)$, and that $f(n)\leq n$. For all $k$ under consideration, that is, for $1 \leq k \leq \min\{(2n\log d)/d,t/2\}$, we have
\begin{eqnarray*}
{t \choose k} \left(\frac{m(k,G^p_n)}{n-k}\right)^{t-k} & \leq & t^k \left(\frac{n(1-p)^k + 3\sqrt{kn(1-p)^k\log n}}{n-\frac{2n\log f(n)}{f(n)}}\right)^{t/2} \\
& = & \left(t(1-p)^{t/4}\right)^k\left(\frac{(1-p)^{k/2}+3\sqrt{(k\log n)/n}}{1-\frac{2\log f(n)}{f(n)}}\right)^{t/2}.
\end{eqnarray*}
Noting that $t \leq n$, we take $t \geq (20\log n)/p$ to get
$$
\left(t(1-p)^{t/4}\right)^k \leq \frac{1}{n^4}.
$$
If we take $p \geq Dn^{-1/2} \log^{1/2}n$ for a suitably large constant $D$ then an easy calculation gives
$$
\frac{(1-p)^{k/2}+3\sqrt{(k\log n)/n}}{1-\frac{2\log f(n)}{f(n)}} \leq 1
$$
for all $k$. It follows that for $p$ in this range, we have $x(t) \leq 1/n^3$ (note that there are at most $n$ summands in (\ref{spec-xt-bound.1})).

In summary, we have
$$
x(t) \leq \left\{
\begin{array}{ll}
\frac{n}{\log^2 n} & \mbox{for $p\geq \frac{\log n + \log\log n + \omega(1)}{n}$ and $t \leq \log n -2\log \log n$} \\
\sqrt{n} & \mbox{for $p\geq \delta$ and $t \leq C'\log n$, for any $\delta, C' >0$} \\
\frac{1}{n^3} & \mbox{for $p\geq Dn^{-1/2}\log^{1/2}n$ and $t \geq \frac{20\log n}{p}$, for large $D>0$}.
\end{array}
\right.
$$
These bounds, together with (\ref{xt-bound-uni}) and (\ref{xt-bound-logcon}), give all parts of Theorem \ref{thm-randomunimodality}.

\section{Proof of Theorem \ref{thm-main-weighted}} \label{Section-proofmain}

We now turn to Theorem \ref{thm-main-weighted}. Recall that the graphs $G$ that we deal with in this section all have $n$ vertices and minimum degree at least $\delta$, and our aim is to maximize the number of independent sets in $G$ subject to these conditions.

Let $I$ be the vertex set of an independent set of maximum size in $G$, and let $J$ be the set of vertices not in $I$. If $|I| \geq n-\delta$ then since $|J| \leq \delta$ and vertices from $I$ can only be adjacent to vertices from $J$, in order to satisfy the condition that $G$ has minimum degree at least $\delta$ we must have $|I|=n-\delta$ and every vertex in $I$ adjacent to every vertex in $J$. We then have
$$
P(G,x) = (1+x)^{n-\delta} + P(G[J],x) - 1
$$
where $G[J]$ is the subgraph of $G$ induced by $J$. This quantity is maximized uniquely when $G[J]$ has no edges, that is when $G=K_{\delta, n-\delta}$.

From now on we assume that $|I| = n -k \leq n -(\delta +1)$, and we will show that in this case $P(G,x) \leq (1+x)^{n-\delta} < P(K_{\delta,n-\delta},x)$. (Note that $P(K_{\delta,n-\delta},x)=(1+x)^{n-\delta} + (1+x)^\delta -1$.) By (\ref{Alekseev-bd}) we have
$$
P(G,x) \leq (1+x)^{n-k} e^{\frac{kx}{1+x}} \leq (1+x)^{n-\delta}
$$
with the second inequality above valid as long as $k \geq C_x\delta$. So we may further assume that $\delta + 1 \leq k < C_x\delta$. (Note that the right-hand side of (\ref{cond}) may be rewritten as
$$
C_x\delta + (\delta - D_x)((C_x-1)\delta +1) > C_x \delta,
$$
so for the range of $n$ we are considering, all $k$ in the range $[\delta + 1, C_x\delta)$ make sense.)

Our strategy now is to show that there are many vertices in $J$ that have large degree to $I$. This limits the contribution to $P(G,x)$ from independent sets whose intersection with $J$ is too large.

Since $G[I]$ is empty (has no edges), the number of edges from $I$ to $J$ (and so also from $J$ to $I$) is at least $\delta(n - k)$. Let $J=\{v_1, \ldots, v_k\}$ and let $a_i$ denote the number of edges from $v_i$ to $I$ for each $i$. Without loss of generality we may assume that $a_1 \geq a_2 \geq \ldots \geq a_k$. We have
$$
\delta(n - k) \leq \sum_{i=1}^k a_i \leq (k-\delta+1)a_{\delta} + (\delta-1)(n-k)
$$
since each $a_i \leq n-k$, and so
\begin{equation} \label{knockout}
a_\delta \geq \frac{n-k}{k-\delta+1}.
\end{equation}
Set $J' = \{v_1, \ldots, v_\delta\}$. We will upper bound $P(G,x)$ by first considering those independent sets which have empty intersection with $J'$, and then those which have non-empty intersection.

Each $S \subseteq J\setminus J'$ with $G[S]$ empty has $|N_I(S)| \geq |S|$, where $N_I(S)$ is the set of vertices in $I$ adjacent to something in $S$ (since otherwise $(I \cup S) \setminus N_I(S)$ would be a larger independent set in $G$ than $I$). We therefore have
\begin{eqnarray}
P(G[I \cup (J\setminus J')],x) & = & \sum_{S \subseteq J\setminus J'\,:\,\, G[S]~\mbox{empty}} x^{|S|}(1+x)^{n-k-|N_I(S)|} \nonumber \\
& \leq & (1+x)^{n-k} \sum_{S \subseteq J\setminus J'} \left(\frac{x}{1+x}\right)^{|S|} \nonumber \\
& = & (1+x)^{n-k} \left(\frac{1+2x}{1+x}\right)^{k-\delta} \nonumber \\
& \leq & (1+x)^{n-\delta} \left(\frac{1+2x}{(1+x)^2}\right)\label{case1}.
\end{eqnarray}
It remains to consider $P(G,x) - P(G[I \cup (J\setminus J')],x)$. By (\ref{knockout}), each non-empty subset $S$ of $J'$ that induces an empty graph has $|N_I(S)| \geq (n-k)/(k-\delta+1)$. It follows that
\begin{eqnarray}
P(G,x) - P(G[I \cup (J\setminus J')],x) & \leq & (1+x)^\delta (1+x)^{n-\delta - \frac{n-k}{k-\delta+1}} \nonumber \\
& \leq & (1+x)^{n-\delta} (1+x)^{\delta - \frac{n-C_x\delta}{C_x\delta-\delta+1}}. \label{case2}
\end{eqnarray}
Combining (\ref{case1}) and (\ref{case2}) we find that
$$
P(G,x) \leq (1+x)^{n-\delta} \left(\frac{1+2x}{(1+x)^2} + (1+x)^{\delta - \frac{n-C_x\delta}{C_x\delta-\delta+1}}\right) \leq (1+x)^{n-\delta},
$$
the second inequality valid by our hypothesis relating $n, x$ and $\delta$.


\begin{thebibliography}{99}

\bibitem{AlaviErdosMaldeSchwenk}
Y. Alavi, P. Erd\H{o}s, P. Malde and A. Schwenk, The vertex independence sequence
of a graph is not constrained, {\em Congressus Numerantium} {\bf 58} (1987), 15--23.

\bibitem{Alekseev}
V. Alekseev, The Number of Maximal Independent Sets in Graphs from Hereditary Classes, in
Combinatorial-Algebraic Methods in Discrete Optimization, (Izd-vo Nizhegorodskogo Un-ta, Nizhnii Novgorod, 1991), 5–-8.

\bibitem{AlonKrivelevich}
N. Alon and M. Krivelevich, The Choice Number of Random Bipartite Graphs, {\em Annals of Combinatorics} {\bf  2} (1998), 291--297.

\bibitem{AlonSpencer} N. Alon and J. Spencer, {\em The
Probabilistic Method}, Wiley, New York, 2000.

\bibitem{ChudnovskySeymour}
M. Chudnovsky and P. Seymour,
The Roots of The Stable Set Polynomial of a Claw-free Graph, {\em J. Combin. Theory. Ser. B} {\bf 97} (2007), 350--357.

\bibitem{CutlerRadcliffe}
J. Cutler and A. Radcliffe,
Extremal graphs for homomorphisms, {\em J. Graph Theory}, to appear.

\bibitem{ErdosRenyi2}
P. Erd\H{o}s and A. R\'enyi, On random matrices, {\em Publ. Math. Inst. Hungar. Acad. Sci.} {\bf 8} (1964), 455--461.

\bibitem{ErdosRubinTaylor}
P. Erd\H{o}s, A. Rubin and H. Taylor, Choosability in graphs, {\em Congressus Numerantium XXVI} (1979),
125-–157.

\bibitem{Frieze}
A. Frieze,
Limit distribution for the existence of Hamilton cycles in random bipartite graphs, {\em Europ. J. Combinatorics} {\bf 6} (1985), 327--334.

\bibitem{Galvin-Qdfixed}
D. Galvin, The independent set sequence of regular bipartite graphs, submitted.

\bibitem{Galvin-asymptot-ind-count}
D. Galvin, An upper bound for the number of independent sets in regular graphs, {\em Discrete Math.} {\bf 309} (2009), 6635--6640.

\bibitem{GutmanHarary}
I. Gutman and F. Harary, Generalizations of the matching polynomial, {\em Utilitas Mathematica}
{\bf 24} (1983), 97--106.

\bibitem{Hamidoune}
Y. Hamidoune, On the number of independent $k$-sets in a claw-free graph, {J. Combin. Theory B} {\bf 50} (1990), 241--244.

\bibitem{HeilmannLieb}
O.~Heilmann and E. Lieb,
Theory of monomer-dimer systems,
{\em Comm. Math. Phys.} {\bf 25} (1972), 190-–232.

\bibitem{Hua}
H. Hua. A sharp upper bound for the number of stable sets in graphs with given
number of cut edges, {\em Applied Mathematics Letters} {\bf 22} (2009), 1380--1385.

\bibitem{Kahn}
J. Kahn, An Entropy Approach to the Hard-Core Model on Bipartite Graphs,
{\em Combin. Probab. Comput.} {\bf 10} (2001),
219--237.

\bibitem{LevitMandrescu}
V. Levit and E. Mandrescu, Partial unimodality for independence polynomials of K\"{o}nig-Egerv\'{a}ry graphs,
{\em Congr. Numer.} {\bf 179} (2006), 109--119.

\bibitem{LinLin}
S. Lin and C. Lin, Trees and forests with large and small independent indices,
{\em Chin. J. Math.} {\bf 23} (1995), 199--210.

\bibitem{MerrifieldSimmons}
R. Merrifield and H. Simmons, {\em Topological Methods in Chemistry}, Wiley, New York, 1989.

\bibitem{PedersenVestergaard}
A. Pedersen and P. Vestergaard, Bounds on the number of vertex independent sets in a graph, {\em Taiwanese Journal of Mathematics} {\bf 10} (2006), 1575--1587.

\bibitem{ProdingerTichy}
H. Prodinger and R. Tichy, Fibonacci numbers of graphs, {\em The Fibonacci Quarterly}
{\bf 20} (1982), 16--21.

\bibitem{Wingard}
G. Wingard, Properties and applications of the Fibonacci polynomial of a graph, Ph.D. thesis, University of Mississippi, May 1995.

\bibitem{Zhao}
Y. Zhao, The Number of Independent Sets in a Regular Graph, {\em
Combin. Probab. Comput.} {\bf 19} (2010), 315--320.

\end{thebibliography}
\end{document}